\documentclass[12pt]{amsart}

\usepackage{amsfonts}
\usepackage{amssymb}
\usepackage{amsthm}
\usepackage{amsmath}
\usepackage{enumitem}
\usepackage{graphicx}
\newtheorem{theorem}{Theorem}[section]
\newtheorem{lemma}[theorem]{Lemma}
\newtheorem{remark}[theorem]{Remark}
\newtheorem{proposition}[theorem]{Proposition}

\newtheorem{definition}[theorem]{Definition}
\numberwithin{equation}{section}

\begin{document}
\title{Free boundary minimal surfaces of unbounded genus}
\author{Daniel Ketover}\address{Department of Mathematics\\Princeton
University\\Princeton, NJ 08544}
\thanks{The author was partially supported by NSF-PRF DMS-1401996 as well as ERC-2011-StG-278940}
 \email{dketover@math.princeton.edu}
\maketitle
\begin{abstract}
For each integer $g\geq 1$ we use variational methods to construct in the unit $3$-ball $B$ a free boundary minimal surface $\Sigma_g$ of symmetry group $\mathbb{D}_{g+1}$.  For $g$ large, $\Sigma_g$ has three boundary components and genus $g$.  As $g\rightarrow\infty$ the surfaces $\Sigma_g$ converge as varifolds to the union of the disk and critical catenoid.  These examples are the first with genus greater than $1$ and were conjectured to exist by Fraser-Schoen.  We also construct several new free boundary minimal surfaces in $B$ with the symmetry groups of the cube, tetrahedron and dodecahedron.  Finally, we prove that free boundary minimal surfaces isotopic to those of Fraser-Schoen can be constructed variationally using an equivariant min-max procedure.  We also prove an $\epsilon$-regularity theorem for free boundary minimal surfaces in $B$. 
\end{abstract}
\section{Introduction}
Denote by $U$ an open domain in $\mathbb{R}^{n+1}$.  A hypersurface $\Sigma^{n}\subset U$ with $\partial\Sigma\subset\partial U$ is called a \emph{free boundary minimal surface} if it is minimal in $U$ and
\begin{equation}
\Sigma\perp\partial U.
\end{equation}
Free boundary minimal surfaces arise variationally as critical points to the volume functional for hypersurfaces with boundary in $\partial U$ where one permits variations that move $\partial\Sigma$ within $\partial U$.  They have been studied already since the 1940s by Courant \cite{C}.  Recently Fraser-Schoen \cite{FS} have found some connections between free boundary minimal surfaces and extremal metrics for Steklov eigenvalues.   

As for existence theory, Gr\"uter-Jost (\cite{GJ},\cite{GJ2}) used the min-max method pioneered by Almgren-Pitts \cite{pitts} and Simon-Smith \cite{ss} to produce a free boundary minimal disk in three-dimensional convex bodies.  Dropping the convexity assumption, Li  \cite{li} recently obtained a related existence result.  Applying White's degree theory \cite{white}, Maximo-Nunes-Smith \cite{mns} showed the existence of a free boundary minimal annulus in convex bodies.  De Lellis-Ramic \cite{DR} have produced free boundary surfaces in higher dimensions (see also \cite{lz}).  See \cite{FGM} for explicit higher dimensional examples.

In this paper, we will be concerned with free boundary minimal surfaces in the standard three-ball $B$ in $\mathbb{R}^3$.  Very few explicit examples are known.  The simplest examples are the the flat disks through the origin.  Secondly, there is the ``critical catenoid" which is a free boundary annulus obtained by rescaling the catenoid in $\mathbb{R}^3$ so that it intersects the boundary of the unit ball orthogonally.  Throughout this paper, $D$ will denote the unique free boundary minimal disk in the $xy$ plane, and $C$ the unique critical catenoid that is rotationally symmetric about the $z$ axis.

For each $k\geq 2$, Fraser-Schoen \cite{FS} constructed a free boundary minimal surface $F_k$ resembling a ``doubling" of the flat disk $D$ in the sense that $F_k\rightarrow 2D$ as $k\rightarrow\infty$.  The surfaces $F_k$ consist of two disks joined by many half-necks at the boundary $\partial B$ and so have genus $0$ and $k$ ends.  Later Folha-Pacard-Zolotareva \cite{pacard} using gluing methods gave another construction of $F_k$ when $k$ is large.  They also constructed related genus $1$ examples $G_k$ (when $k$ is large) by adding a catenoidal neck at the center of the $F_k$ examples joining the two layers.  

It has been an open question whether one can construct higher genus examples.
Fraser-Schoen conjectured (Section 1 in \cite{FS2}) that there should be a sequence of free boundary minimal surfaces with genus approaching infinity converging to the union of the critical catenoid and the disk.  In this paper, we confirm their conjecture and prove: 
\begin{theorem}\label{main}
For each integer $g\geq 1$, there exists a free boundary minimal surface $\Sigma_g$ in the unit $3$-ball $B$ with dihedral symmetry $\mathbb{D}_{g+1}$ that is not the flat disk $D$.  When $g$ is large, $\Sigma_g$ has three boundary components and genus $g$.  Moreover, 
\begin{equation}\label{limit}
\Sigma_g\rightarrow D\cup C \text{ in the varifold sense as } g\rightarrow\infty.
\end{equation}
Furthermore, $|\Sigma_g|< |D|+|C|$ for all $g\geq 1$ (where $|\Sigma|$ denotes the $2$-dimensional Hausdorff measure of $\Sigma$).
\end{theorem}
The proof of Theorem \ref{main} was inspired by a sketch of Pitts-Rubinstein \cite{PR} for a variational construction of the minimal surfaces of Costa-Hoffman-Meeks (see also \cite{HM}).  The minimal surfaces obtained by Theorem \ref{main} can be interpreted as free boundary analogs to these surfaces.  Surprisingly, in the proof of Theorem \ref{main} we need the sharp isoperimetric inequality for the ball $B$ to prove that the sweepouts we consider are nontrivial.

When $g$ is small, it is possible that the surface $\Sigma_g$ produced by Theorem \ref{main} has one boundary component in $\partial B$ rather than three.   One could rule this out if one could show that the only free boundary minimal surface in $B$ with one boundary component is a flat disk.   This would imply the existence of a free boundary minimal surface of each genus $g$.

We also apply the catenoid estimate \cite{KMN} to show:
\begin{theorem}\label{fs}
For each $k\geq 2$, a free boundary minimal surface isotopic to $F_k$ can be constructed variationally through a one-parameter equivariant min-max procedure.  
\end{theorem}
%\begin{remark}\normalfont
%Note that \cite{pacard} only construct $G_k$ when $k$ is large.  Theorem \ref{fs} extends their existence result for all %$k$.
%\end{remark}
%Finally, we produce several new free boundary minimal surfaces associated to the Platonic solids:
Finally we produce several new genus zero examples associated with the Platonic solids: 

\begin{theorem}\label{platonic}
There exists a free boundary minimal surface in $B$  with octahedral symmetry of genus $0$ and $6$ ends, an example with tetrahedral symmetry of genus $0$ and $4$ ends, and an example of genus $0$ and $12$ ends of dodocahedral symmetry.
\end{theorem}

Let us describe the genus $0$ surface with $6$ ends produced by Theorem \ref{platonic}.  Consider the graph $\mathcal{G}$ in $B$ consisting of the union of the $x$, $y$ and $z$ axes in $B$.  The free boundary minimal surface we construct is isotopic to the boundary of a tubular neighborhood of $\mathcal{G}$, resembling a three-dimensional ``cross."

Since the methods involved in proving Theorem \ref{main} arise from a global variational principle, they are quite versatile and apply in other ambient geometries.  They can be used to give a min-max construction of the self-shrinkers discovered by Kapouleas-Kleene-Moller \cite{KKM} and independently Nguyen \cite{Ng}:

\begin{theorem} (Kapouleas-Kleene-M{\o}ller \cite{KKM}, Nguyen \cite{Ng})\label{th4}
For $g$ large enough, there exist a self-shrinker $N_g$ with dihedral symmetry $\mathbb{D}_{2(g+1)}$ (acting by rotations about the $z$-axis) with one end and genus $g$.
As $g\rightarrow\infty$, $N_g$ converge as varifolds to the union of the self-shrinking $xy$-plane and self-shrinking sphere.
\end{theorem}

While our existence theorem could produce examples of low genus with the symmetries of those of Theorem \ref{th4} when $g$ is small, we would need to take $g$ large to rule out that the min-max limit we obtain has many ends. 

In gluing constructions, one always needs to take the genus along the desingularizing curves to be large.  One advantage of the variational approach is that one can produce minimal surfaces with low genus as in Theorem \ref{main}.  However, one still has to rule out various compressions (``neck-pinches") that occur as the min-max sequence converges to its limiting minimal surface.  Thus uniqueness or classification theorems are still needed when the genus is low to rule out unwanted behavior such as the surfaces $\Sigma_g$ having genus greater than zero but one end.  When the genus is large, typically one can take a limit in the underlying symmetry group to see that the limiting stationary varifold one produces has enough symmetries to be easily classifiable (see for instance the proof of Theorem \ref{main} or Theorem 1.7 in \cite{KMN}).

On a related note, using higher parameter families, Marques-Neves \cite{mn} proved that manifolds with positive Ricci curvature contain infinitely many embedded minimal surfaces.   Li-Zhou \cite{lz} have adapted their argument to the free boundary setting.  Of course, manifolds with many symmetries such as the $3$-ball trivially contain infinitely many minimal free boundary surfaces by considering all rotations of the disk or critical catenoid.   Thus it is not so clear how to use Marques-Neves' method to obtain new minimal surfaces.  Recently Aiex \cite{Ai} has shown that the set of min-max surfaces produced by the Marques-Neves procedure is non-compact, which should rule out that the infinitely many surfaces produced are simply rotations of the basic ones.   Still, these methods would not allow any precise control on the topological type.

Let us explain the geometric idea behind the existence of the surfaces $\Sigma_g$ purported by Theorem \ref{main}.  They arise via an equivariant min-max procedure. The sweepouts we consider are somewhat unusual in that they do not start and end at surfaces with zero area as one usually has in min-max theory.  Instead, they begin and end at the unit disk $D$ (thus one can alternatively think of them as being parameterized by $\mathbb{S}^1$ instead of $[0,1]$).  Let us describe these sweepouts. There's an optimal sweepout $C_t$ of $B$ by annuli orthogonal to $D$ with the critical catenoid sitting in the middle of the foliation.  We consider the family $\Sigma_t=C_t\cup D$, desingularized by adding in genus $g$ in a $\mathbb{D}_{g+1}$-equivariant manner along the circle of intersection between $D$ and $C_t$.  By the $\mathbb{D}_{g+1}$-equivariance, each surface in this family divides $B$ into two components of equal volume. By Almgren's solution \cite{Al} of the isoperimetric problem for $B$, it follows that each surface in $\Sigma_t$ has area at least $\pi$. The key observation is that the sweepout $\Sigma_t$ (and any sweepout in the equivariant saturation of $\Sigma_t$) \emph{interchanges} the two hemispheres of $B$ determined by $D$.  

One then considers the width $W$ for the min-max problem for sweepouts in the equivariant saturation of $\Sigma_t$.  In order to show that this sweepout is non-trivial, we must show that $W$ is greater than the area of the disk, $\pi$.  If $W=\pi$, then by definition of width, one could find sweepouts all of whose slices have area very close to $\pi$.  By Almgren's result again \cite{Al}, it follows that all of these slices are very close as varifolds to the disk $D$ as $D$ is the only equivariant surface solving the isoperimetric problem.  But one can show that such sweepouts cannot swap the two hemispheres.

While generalizing the min-max theory to the free boundary setting has already been carried out (\cite{DR}, \cite{lz}, \cite{LZ2}), from a min-max perspective, there are some new ingredients that are needed in Theorem \ref{main}.   In min-max theory, the regularity of the surface produced is proved by showing that the min-max sequence is approximated locally by stable surfaces and using the compactness theorem for such surfaces due to Schoen \cite{schoen}.  

In \cite{ketover2}, this theory was extended to considering sweepouts satisfying a symmetry.  There one restricts to sweepouts of $G$-equivariant surfaces for some group of isometries $G$ so that each surface in the sweepout intersects the singular set of the group action transversally.  One shows there (Section 4 in \cite{ketover2}) that one can produce $G$-stable replacements (i.e. minimal surfaces that are stable among $G$-equivariant deformations) and moreover that $G$-stability is equivalent to stability.  Thus Schoen's compactness result applies.

 In the setting of Theorem \ref{main}, there is a key difference.  The min-max sequence contains segments of the singular set of the action of $\mathbb{D}_{g+1}$ on $B$.  In balls about such segments one can find approximating $\mathbb{Z}_2$-stable surfaces.  But because the surface contains the axis, $\mathbb{Z}_2$-stability \emph{does not} seem to imply stability.  Thus we need to replace Schoen's estimates in the regularity theory with the fact that minimal surfaces of bounded area and genus have a convergent subsequence \cite{CS}.  The convergence may be non-smooth over finitely many points, but this does not impede the regularity theory or genus bounds.  The arguments thus are inherently two-dimensional and only apply to the Simon-Smith theory \cite{ss}.

The organization of this paper is as follows.   In Section \ref{rig} we prove an $\epsilon$-regularity type theorem expressing the rigidity of the free boundary disk.  In Section \ref{minmax} we state the equivariant min-max theorem we need in this paper in the free boundary setting.  In Section \ref{con} we prove Theorem \ref{main}.  In Section \ref{fraserschoen} we construct the examples of Fraser-Schoen, and in Section \ref{platonics} the Platonic examples of genus $0$.  Finally in Section \ref{proof} we prove the equivariant min-max theorem.

It has been brought to my attention that N. Kapouleas and M. Li have applied gluing methods to obtain an analog of Theorem \ref{main} when $g$ is large.

\emph{Acknowledgements:} I thank Brian White for a conversation and Otis Chodosh for bringing the work of Volkmann \cite{V} to my attention.  I thank Fernando Marques and Andr\'e Neves for discussions and encouragement.

\section{Rigidity of free boundary minimal disk}\label{rig}
We need the following strong rigidity statement for the free boundary minimal disk which may be of independent interest:
\begin{proposition}\label{rigidity}
There exists $\epsilon>0$ so that if $\Sigma$ is an embedded free boundary minimal surface in $B$ and
\begin{equation}
|\Sigma| < \pi + \epsilon,
\end{equation}
then $\Sigma$ is a flat free boundary disk $D$.
\end{proposition}
One can think of Proposition \ref{rigidity} as strong form of Allard's regularity theorem \cite{A} for free boundary minimal surfaces that holds \emph{up the boundary.}  Indeed, Allard's theorem loosely speaking gives that a minimal surface in the ball with area close enough to the flat disk is a graph over this flat disk in a sub-ball.   Adding the free boundary condition to these assumptions, Proposition \ref{rigidity} gives a stronger conclusion as one obtains graphicality up to the boundary and moreover, one gets a unique graph, i.e., the flat disk itself.  This improvement is one of the special features of free boundary minimal surfaces. For related (sharp) gap theorems for the free boundary minimal disk, see \cite{AN}. 

The proof of Proposition \ref{rigidity} uses a monotonicity formula for free boundary minimal surfaces that was discovered in various guises independently by Fraser-Schoen (Theorem 5.4 in \cite{FS}), Ros-Vergasta \cite{RV}, Brendle \cite{B}, and obtained in most general form by Volkmann \cite{V}.

To state the consequence of Volkmann's formula that we need, let $\Sigma$ be a free boundary minimal surface in $B$.   For any $x_0$ in the support of $\Sigma$, there holds (by plugging $H=0$ into Theorem 3.4.3 and using equation (4) in \cite{V}):
\begin{equation}\label{volk}
\int_\Sigma\frac{|(x-x_0)^\perp |^2}{|x-x_0|^4} \leq |\Sigma|-|D|.
\end{equation}
The vector $(x-x_0)^\perp$ is the component of the vector $x-x_0$ that is othogonal to the tangent plane $T_{x_0}\Sigma$.   Volkmann's monotonicity formula compares a quantity at two scales --  letting one scale approach $0$, and the other approach $\infty$ in his formula one obtains \eqref{volk}.  

One easy consequence of \eqref{volk} is the fact that $D$ has the smallest area among free boundary minimal surfaces, and is the unique surface with this property (cf. \cite{B}).  The integrated term in \eqref{volk} 
\begin{equation}
E(x_0,\Sigma):=\int_\Sigma\frac{|(x-x_0)^\perp |^2}{|x-x_0|^4}
\end{equation}
is a kind of ``tilt-excess," that controls the deviation of $\Sigma$ from a plane.  A convenient feature is that the point $x_0$ in \eqref{volk} is arbitrary, while the right hand side of \eqref{volk} does not depend on $x_0$.  A key point is that $E(x_0,\Sigma)$ is \emph{scale invariant} if $\Sigma$ is rescaled about $x_0$.  In other words 
\begin{equation}\label{scale}
E(x_0,\Sigma)=E(0,\tau(\Sigma-x_0)) \mbox{ for any } \tau > 0.
\end{equation}
\\
\noindent
 \emph{Proof of Proposition \ref{rigidity}}\\
We first establish the following claim:

\emph{claim: There exists $\epsilon>0$ and $C>0$ so that whenever $\Sigma$ is an embedded free boundary minimal surface with $|\Sigma|\leq\pi + \epsilon$ there hold the following curvature bounds
\begin{equation}
\sup_{x\in\overline{\Sigma}}|A|^2(x) \leq C.
\end{equation} }
\\
This claim easily implies Proposition \ref{rigidity}.  Indeed, suppose the claim is true.  If Proposition \ref{rigidity} failed, it means there is a sequence of free boundary surfaces that are different from flat disks, with areas approaching $\pi$ and by the claim, uniformly bounded curvature.  Thus by Theorem 5.1 in \cite{ML}, $\Sigma_i$ converge smoothly (up to the boundary) to a disk (see also Section 5 of \cite{LZ2}).  Since the convergence is smooth, it follows that $\Sigma_i$ are all disks for $i$ large enough.  But by a theorem of Nitsche \cite{N}, the only free boundary minimal disks are the flat ones.  This is a contradiction. 

It remains to prove the claim.  We will use a blowup argument (cf. \cite{white2}).  Suppose the claim is false.  Thus there is a sequence $\Sigma_i$ of free boundary minimal surfaces in $B$ where
\begin{equation}\label{areas22}
|\Sigma_i|-\pi\rightarrow 0
\end{equation}
while 
\begin{equation}\label{toinf}
A_i :=\sup_{x\in\Sigma_i}|A|^2(x) \rightarrow \infty.
\end{equation}
For each $i$, choose a point $x_i\in\Sigma_i$ where the $\sup$ $A_i$ is attained in \eqref{toinf} (of course $x_i$ may be in the boundary of $\Sigma_i$).  Then consider the sequence of surfaces
\begin{equation}
\tilde{\Sigma}_i := A_i(\Sigma_i - x_i),
\end{equation}
which satisfy
\begin{equation}\label{bounds}
\sup_{x\in\tilde{\Sigma}} |A|^2 \leq 1 
\end{equation}
and
\begin{equation}\label{still}
|A|^2_{\tilde{\Sigma}_i}(0)=1.
\end{equation}

Since dilations preserve angles, the surface $\tilde{\Sigma}_i$ still satisfies the free boundary condition for the new domain $A_i(B-x_i)$.

By the curvature bounds \eqref{bounds} it follows that $\tilde{\Sigma}_i\rightarrow \Sigma_\infty$ smoothly.  Note that $\Sigma_\infty$ is either complete without boundary or (after rotating $\Sigma_j$ potentially) is contained in a half-space and its boundary is contained in a plane.  The convergence in this latter case is smooth up to the boundary by the  boundary Schauder estimates of Agmon-Douglis-Nirenberg (Theorem 9.1 in \cite{Ag}) because the free boundary condition implies that the surfaces locally near the boundary satisfy an elliptic equation with homogeneous boundary conditions.  By the smooth convergence $\Sigma_j\rightarrow\Sigma_\infty$ (up to the boundary if it exists) and \eqref{still} we still have:
\begin{equation}\label{yay} 
|A|^2_{\Sigma_\infty}(0)=1.
\end{equation}
In light of \eqref{areas22} and \eqref{volk}, we obtain
\begin{equation}
E(x_0,\Sigma)\rightarrow 0.
\end{equation}
By the scale invariance \eqref{scale}, we obtain
\begin{equation}
E(0,\tilde{\Sigma}_i)\rightarrow 0.
\end{equation}
Since the convergence $\tilde{\Sigma}_j\rightarrow \Sigma_\infty$ is smooth, we have
\begin{equation}
0\leq E(0,\Sigma_\infty)\leq \liminf E(0,\tilde{\Sigma}_i) =0.
\end{equation}
Since $E(0,\Sigma_\infty)=0$, it follows that $\Sigma_\infty$ is a plane or half-plane.  But this violates \eqref{yay}.  Thus the claim is established.  
\qed

\section{Equivariant min-max theory}\label{minmax}
In \cite{ketover2}, an equivariant min-max theory was developed to produce min-max minimal surfaces with symmetries. 
Throughout this paper $G$ will denote either $\mathbb{Z}_n$, $\mathbb{D}_n$ or one of the three groups associated to the Platonic solids $T_{12}$, $O_{24}$, or $I_{60}$ acting standardly on $B$.

For any $x\in B$ we first define the \emph{isotropy subgroup} $G_x$ at $x$ as:
$$G_x=\{g\in G\;\;|\;\;gx = x\}$$
We then define the \emph{singular locus} of the group action as points with nontrivial isotropy subgroup:
$$\mathcal{S}=\{ x\in B\;\;|\;\;G_x\neq\{e\}\}$$ 

The set $\mathcal{S}$ consists of straight line segments emanating from the origin.
In \cite{ketover2}, one considers sweepouts of $B$ that intersect $\mathcal{S}$ transversally.  But the sweepouts we consider in Section \ref{con} to prove Theorem \ref{main} are not transverse to the singular set -- they \emph{contain} line segments which are part of $\mathcal{S}$.  Thus the theory developed in \cite{ketover2} needs to be extended to that setting.  We make the following more general definition:

A {\it (genus $g$) $G$-sweepout} of $B$ is a family of closed sets $\{\Sigma_t\}_{t=0}^1$, continuously varying in the Hausdorff topology such that:
\begin{enumerate}[label=\roman*.,topsep=2pt,itemsep=-1ex,partopsep=1ex,parsep=1ex]
\item $\Sigma_t$ is a smooth embedded surface of genus $g$ for $0<t<1$ varying smoothly
\item $\Sigma_0$ and $\Sigma_1$ are the union of a smooth surface together with a collection of arcs
\item Each $\Sigma_t$ is $G$-equivariant, i.e. $g(\Sigma_t)=\Sigma_t$ for $0\leq t\leq 1$ and all $g\in G$
\end{enumerate}
\begin{remark}
Note that if $\Sigma_t$ contains an arc $\mathcal{A}\subset\mathcal{S}$ for some $t$, then $\mathcal{A}$ is in the support of  $\Sigma_t$ for \emph{all} $t$.  This follows from the fact the surfaces vary smoothly and one cannot smoothly ``push off the axis."  See Lemma 3.6 in \cite{ketover2}.
\end{remark}

Given such a $G$-sweepout $\{\Sigma_t\}_{t=0}^1$ we may define the $G$-equivariant saturation $\Pi=\Pi_{\{\Sigma_t \}}$ identically as in \cite{ketover2}.  We can then define the min-max width:
\begin{equation}\label{inf}
W^G_\Pi = \inf_{\{\Lambda_t\}\in\Pi}\sup_{t\in [0,1]}|\Lambda_t|.
\end{equation}
\noindent

We can then consider a sequence of sweepouts $\{\Sigma_t\}^i$ the area of whose maximal slice converges to $W^G_\Pi$.  From $\{\Sigma_t\}^i$ we may then choose a sequence of slices $\Sigma_i:=\Sigma_{t_i}^i$ with area converging to $W^G_\Pi$.  Such a sequence of surfaces we will call a \emph{min-max sequence}. 

With this notation we have the following Min-Max theorem (where the relevant terms in ii) are explained in the remarks following the theorem):
\begin{theorem}\label{eqminmax} 
If \begin{equation}\label{isnontrivial}
W^G_\Pi > \max(|\Sigma_0|,|\Sigma_1|)
\end{equation}
then there exists a min-max sequence $\Sigma_j$ converging as varifolds to $n\Gamma$, where $\Gamma$ is a smooth embedded connected free boundary minimal surface in $B$ and $n$ is a positive integer.  Moreover, the following statements hold:
\begin{enumerate}[label=\roman*.] 
\item $W^G_\Pi=n|\Gamma|$
\item For $j$ large enough, after performing finitely many surgeries on $\Sigma_j$ and discarding some components, each remaining component of $\Sigma_j$ is isotopic to $\Gamma$ and there are $n$ such components.  Thus 
\begin{equation}
n(\mbox{genus}(\Gamma))\leq g.
\end{equation}
%\item Item c) implies the genus bound with multiplicity:
%\begin{equation}\label{multi}
%\sum_{i\in\mathcal{O}} n_ig(\Sigma_i)+\sum_{i\in\mathcal{N}}\frac{n_i}{2}(g(\Sigma_i)-1)\leq g,
%\end{equation}
%where $\mathcal{O}$ denotes the subcollection of $\Gamma_i$ that are orientable, and $\mathcal{N}$ denotes the %subcollection that are non-orientable.
\item If the surfaces $\Sigma_t$ all contain a segment of isotropy $\mathbb{Z}_2$, then $\Gamma$ contains this segment as well and $n$ is odd.
\item If the surfaces $\Sigma_t$ are orthogonal to or are disjoint from an arc $\mathcal{A}\subset\mathcal{S}$ of isotropy $\mathbb{Z}_2$, then $\Gamma$ either intersects $\mathcal{A}$ orthogonally (if at all) or else contains $\mathcal{A}$ and $n$ is even.
\item $\Gamma$ intersects arcs of isotropy $\mathbb{Z}_n$ (for $n\neq 2$) orthogonally (if at all).
\end{enumerate}
\end{theorem}

\begin{remark}\normalfont
The connectedness of $\Gamma$ follows from the Frankel-type property of embedded free boundary minimal surfaces in convex bodies: any two must intersect (Lemma 2.4 in \cite{ML}).
\end{remark}
\begin{remark}\normalfont\label{neckpinches}
By ``surgeries" in ii) is meant either ``$G$-equivariant neckpinch" (cf. Remark 1.6 in \cite{ketover2}) or ``collapse of topology" which we will define below.   A ``$G$-equivariant neckpinch," is obtained in three possible ways: 1) by removing a cylinder and gluing in two disks so that the cylinder and disks as well as the ball they bound is disjoint from $\mathcal{S}$, 2) removing a cylinder centered around $\mathcal{S}$ and adding in two disks, where each disk intersects $\mathcal{S}$ once, or as a ``half neckpinch" at the boundary.   This third type of surgery is needed because of the free boundary and is defined as follows.  Let $\Sigma$ be a surface in $B$ with $\partial\Sigma\subset\partial B$.  Consider a disk $D\subset\Sigma$ with $\partial D = \cup_{i=1}^4\gamma_i$ where $\gamma_i$ are smooth arcs and the four arcs are concatenated in increasing order to give $\partial D$.   The arc $\gamma_1$ is contained in $\partial B$, $\gamma_2$ is contained in the interior of $B$, $\gamma_3$ contained in $\partial B$ and $\gamma_4$ is contained in the interior of $B$.  The arcs $\gamma_2$ and $\gamma_4$ bound the disks $D_1\subset\overline{B}$ and $D_3\subset\overline{B}$ where $\partial D_1 = \gamma_2\cup\alpha_1$ where $\alpha_1$ is an arc contained in $\partial B$.  Likewise $\partial D_3 = \gamma_4\cup\alpha_3$ where $\alpha_3$ is an arc contained in $\partial B$.  The surgery on $\Sigma$ is the removal of $D$ from $\Sigma$ and addition of $D_1$ and $D_2$.   This third type of surgery is loosely speaking a ``half neck-pinch" at the boundary $\partial B$.  

Finally, in the case where the sweepout surfaces $\Sigma_t$ contain an arc $\mathcal{A}'$ of isotropy $\mathbb{Z}_2$, we include ``collapse of topology" in the admissible surgeries.  This consists of the following. Fix $x\in\overline{\mathcal{A'}}$ so that $G_x$ is either $\mathbb{D}_k$ or $\mathbb{Z}_2$.   Fix a small $G_x$-invariant ball $B_x$ about $x$ and let $\Sigma$ be a $G_x$-invariant surface in $B_x$ with $\partial\Sigma\subset\partial B_x$.  We say $\Sigma'$ arises from $\Sigma$ by collapse of topology if $\Sigma'$ is the varifold limit of $G$-equivariant isotopies supported in $B_x$ and if
there is a sequence of neck-pinches supported in $B_x$ (not necessarily $G$-equivariant), so that $\Sigma'$ arises topologically from $\Sigma$ after performing these neck-pinches.  
\end{remark}

\begin{remark}\normalfont\label{throwing}
The reason we include the ``collapse of topology" in the admissible surgeries in ii) is the following.  In the case where the sweepout surfaces contain an axis of the singular set, we only prove that the $\mathbb{Z}_2$-stable replacements $V_j$ around the axis converge smoothly to their smooth limit $V$ away from finitely many points, and near such points it may not be possible to perform $G$-equivariant neck-pinches to obtain $V$. 

As an example of this phenomenon, consider the Costa-Hoffman-Meeks minimal surface $C_g$ of genus $g$ in $\mathbb{R}^3$, that has dihedral symmetry and contains the $\mathbb{Z}_2$ isotropy axes of the symmetry group.  One can consider the family of rescalings $\{\lambda_i C_g\}$ for $\lambda_i\rightarrow 0$ that converges to the plane with multiplicity $3$.  The convergence is smooth away from the origin, where the genus is collapsing.   There is \emph{no way} to perform equivariant neck-pinches on this family to obtain a surface isotopic to three disjoint sheets as any neck-pinch will break the dihedral symmetry.  On the other hand, the configuration consisting of three planes is indeed in the limit of $G$-equivariant isotopies of $C_g$  and it is achievable through surgeries (though not equivariant ones). 
\end{remark}
\begin{remark}\normalfont
While we only state Theorem \ref{eqminmax} in the case where the ambient manifold is a three-ball, since all considerations are local, it is clear that it holds for a general closed three-manifold.
\end{remark}
We defer the proof of Theorem \ref{eqminmax} til Section \ref{proof}.
\section{Proof of Theorem \ref{main}}\label{con}
\subsection{Sweepouts}
Let us first construct the sweepouts we will need for the surfaces constructed in Theorem \ref{main}.  In this section, denote by $D$ the unit disk in the $xy$ plane.

Fix $g\geq 1$.  Consider the group $\mathbb{D}_{g+1}$ acting on $B$ by rotations and of $2\pi/(g+1)$ about the $z$ axis, and also rotations of $\pi$ about the $g+1$ line segments $\{L_i\}_{i=0}^{g}$ (setting $\theta_i = i\pi / (g+1)$): 
\begin{equation}
L_i := \{(r\cos (\theta_i) , r\sin(\theta_i), 0)\in\mathbb{R}^3\;|\;-1\leq r\leq 1\}
\end{equation}

Denote by $C$ the unique critical catenoid with symmetry group $\mathbb{D}_{g+1}$ encircling the $z$-axis.  There exists a sweepout $\{C_t\}_{t=-1}^1$ of $B$ with dihedral symmetry $\mathbb{D}_{g+1}$ with the following properties:

\begin{enumerate}[label=\roman*.] 
\item $C_{-1} = \{\mbox{z-axis}\}\cap B$
\item $C_1 = \partial D$
\item $C_t$ is a smooth annulus for $-1 < t< 1$.
\item $C_{0}$ is the critical catenoid $C$
\item $|C_t| < |C| - At^2$ for some $A>0$
\item $C_t\cap D$ is a round circle $R_t$ for all $-1< t\leq 1$ and $R_t$ sweep-out $D$.
\item $C_t$ is orthogonal to $D$ for  $-1 < t< 1$
\end{enumerate}

Note that the only singular slices in the foliation $C_t$ occur when $t=1$ and $t=-1$ when $C_t$ consists of one dimensional graphs.

In order to construct $C_t$, we start with the critical catenoid $C_0$.  It is stable if one considers deformations that vanish on $\partial B$.  On the other hand, if one considers more general deformations, then its Morse index has recently been computed to be $4$ (\cite{SZ}, \cite{T}, \cite{D}).

  Let $\phi$ be a smooth function defined on $C_0$ and let $n$ be a choice of unit normal on $C_0$.
Then the formula for the second derivative of area is (c.f. Section 5.1 \cite{D}):
\begin{equation}\label{secondvar}
\frac{d^2}{d^2t}\Big|_{t=0}|C_0+tn\phi| = \int_\Sigma |\nabla\phi|^2-|A|^2\phi^2-\int_{\partial\Sigma}\phi^2.
\end{equation}
Setting $\phi_1=1$ in \eqref{secondvar} we obtain
\begin{equation}\label{aaa}
\frac{d^2}{d^2t}\Big|_{t=0}|C_0+tn\phi_1| = -\int_\Sigma |A|^2d\mu-\mathcal{H}^1(\partial\Sigma) <0.
\end{equation}
Thus $\phi_1$, while not an eigenfunction of the stability operator, still gives a rotationally symmetric direction for decreasing the area of the critical catenoid.

%For fixed $\epsilon>0$ small enough, and $|t|\leq\epsilon$, we can now define the foliation
%\begin{equation}
%C_t:= C_0+t\phi_1(x)n(x).
%\end{equation}
%By \eqref{aaa} and Taylor expanding the area functional in powers of $t$ it follows that $C_t$ satisfies (v) for $t$ small.  %Let $R$ be the component of $B\setminus C_\epsilon$ disjoint from $C_0$.  Since $R$ is mean convex, it follows that %for any $\delta>0$, we can extend the sweepout $C_t$ for $t\in[\epsilon,1]$ so that $C_t$ sweep-out the rest of $R$ %and ends at $C_1$ at the round circle $\partial D$ so that 
%\begin{equation}
%|C_t|\leq |C_{\epsilon}|+\delta\mbox{ for } t\in[\epsilon,1]. 
%\end{equation}

%Otherwise (c.f. \cite{MN2}) we obtain from a min-max procedure a free boundary minimal annulus in the interior of $R$ %that is disjoint from $C_0$.   This cannot happen since free boundary minimal surfaces in $B$ are known to all %intersect each other (c.f. Lemma 2.4 in \cite{ML}).  Similarly one can extend $C_{-\epsilon}$ in the other direction to %obtain a sweepout $\{C_t\}_{t=-1}^1$ satisfying i-vii.

Let us now give construction of the optimal sweepout of annuli $C_t$. Let $R$ denote the region of $B\setminus C_0$ that is disjoint from the $z$-axis.  We can consider dilations of the critical catenoid $\lambda C_0$ for $\lambda\geq 1$.  Such dilations preserve the region $R$ in that $\lambda R\subset R$.  For some value $\lambda_0$, $\{\lambda C_0\}_{\lambda=1}^{\lambda=\lambda_0}$ gives a foliation of $R$ interpolating between $C_0$ and $\partial D$.  Since dilating by $\lambda$ takes the ball of radius $B_{1/\lambda}(0)$ to the ball $B$, it follows from the monotonicity formula that the area of $B\cap\lambda C_0$ is a decreasing function in $\lambda$.  This gives the required foliation of $R$.  To fill out $R'=B\setminus R$ we can argue as follows.  First use $\phi_1$ defined above to extend $C_0$ into $R'$ to $C_t$ for $-\epsilon\leq t\leq 0$ by
 \begin{equation}
C_t:= C_0+t\phi_1(x)n(x).
\end{equation}
In light of \eqref{aaa}, $|C_t|\leq |C_0|-At^2$ for some $A>0$.

 Note that the region $R'\setminus\cup_{t\in[-\epsilon,0]}C_t$ is mean convex and the boundary circles of $C_{-\epsilon}$ no longer bound a minimal annulus since the circles $\partial C_{-\epsilon}$ are contained in the region in $\partial B$ where no catenoids can penetrate. Thus there is a path of rotationally symmetric surfaces beginning at $C_{-\epsilon}$ and ending at the two flat disks bounded by $\partial C_{-\epsilon}$ that increases area an arbitrarily small amount along the way.  It is then easy to use parallel disks to $\partial C_{-\epsilon}$ joined by a tiny tube about the $z$-axis to fill out the rest of $R'$.  This gives the required sweepout.

Let us define the singular family of sets for $-1\leq t\leq 1$:
\begin{equation}
\Sigma_t = C_t \cup D.
\end{equation}

We will now amend each surface $\Sigma_t$ in the sweepout in a very small neighborhood of $D\cap C_t$.  The $g+1$ lines $L_i$ intersect $C_t \cap D$ in $2(g+1)$ equally spaced points and divide $C_t\cap D$ into $2(g+1)$ consecutive arcs $\{A_i\}_{i=1}^{2(g+1)}$.   Along $A_1$ are being joined $4$ pieces of smooth surfaces in an ``X": two from $C_t$ and two from $D$.   Desingularize this intersection along $A_1$ by pairing one of the pieces of $C_t$ with one from $D$, and the other $C_t$ piece with the other $D$ piece.  Then proceed to pair off the pieces in the opposite way along $A_2$  and continue in this alternating fashion along $\partial D$ to arrive at a new surface $\tilde{\Sigma}_t$.  Note that to preserve $\mathbb{D}_{g+1}$ symmetry, once the desingularization has been performed along $A_1$, it extends in a unique way to the rest of $\partial D$.  Thus rotating $180^o$ about any of the lines $L_i$ is still a symmetry of $\tilde{\Sigma}_t$.  

 It is clear that we can perform this desingularization so that all changes are supported in $T_{\epsilon(t)}(C_t\cap D)$ for any suitably small continuous positive function $\epsilon(t)$ and we can moreover choose $\epsilon(t)$ to approach $0$ as $t$ approaches $-1$ or $1$. This desingularization is an area-decreasing procedure since desingularizing amounts to ``rounding corners" and thus lowers area.  We can think that each arc $A_i$ is labelled alternatively $+$ or $-$, and combining two consecutive such arcs gives a period of the rotational symmetry of the resulting surface.  Note that $\tilde{\Sigma}_t$ still contains the lines $\{L_i\}_{i=0}^{g}$ since the surfaces are only being adjusted in the interior of each arc $A_i$. 

This desingularization has the effect that $\tilde{\Sigma}_t$ is now a surface of genus $g$ with $\mathbb{D}_{g+1}$ symmetry: rotations of angle $2\pi/(g+1)$ together with such rotations composed with a ``flip" about any of the lines $L_i$.   In summary,  one obtains a sweepout $\tilde{\Sigma}_t$ satisfying the following properties (reparameterizing the sweepout by $[0,1]$ instead of $[-1,1]$):
\begin{enumerate}[label=\roman*.]
\item $|\tilde{\Sigma}_t|< |C|+|D|$ for all $0\leq t\leq 1$
\item $\tilde{\Sigma}_0=D\cup (\{\mbox{z-axis}\}\cap B)$
\item $\tilde{\Sigma}_1 = \partial D$
\item For each $0<t<1$ the genus of $\tilde{\Sigma}_t$ is $g$
\item $\tilde{\Sigma}_t$ has dihedral symmetry $\mathbb{D}_{g+1}$
\item For all $t$, $\tilde{\Sigma}_t$ contains the lines $\{L_i\}_{i=0}^{g}$ and (in particular) the origin 
\end{enumerate}

The sweepout $\{\tilde{\Sigma}_t\}$ satisfies the definition of $\mathbb{D}_{g+1}$-sweepout in Section \ref{minmax}.  Let us denote by $\Pi$ the equivariant saturation of sweepouts containing $\{\tilde{\Sigma}_t\}$.
\begin{remark}
Note that unlike in $2$-dimensions, there are several different types of dihedral symmetry in $3$-dimensions.  The Costa-Hoffman-Meeks surfaces have an extra symmetry coming from reflections in certain planes, and their symmetry group is $\mathbb{D}_{2(g+1)}$.  Our group is only half as large. We do this so that the resulting orbifold $B/G$ has no boundary aside from $\partial B/G$ which is the setting in which equivariant min-max theory was developed in \cite{ketover2}.
\end{remark}

The set $B\setminus D$ has two components.  Denote by $C_1$ the component in the northern hemisphere, and $C_2$ the component in the southern.  The sweepout $\tilde{\Sigma}_t$ has the following key property.  
\begin{enumerate}
\item For each $t$, $\tilde{\Sigma}_t$ divides $B$ into two components $A(t)$ and $B(t)$ of equal volume \item $A(0)=C_1$, $B(0) = C_2$
\item  $A(1)=C_2$, $B(1) = C_1$.
\end{enumerate}
\
In other words, in the course of the sweepout, the two hemispheres of $B$ are swapped.  One can see (1) because rotating $180^o$ through any of the lines $L_i$ acts by isometry to preserve the surface $\tilde{\Sigma}_t$ while interchanging the components $A(t)$ and $B(t)$. 

Moreover, any $\{\Lambda_t\}\in\Pi$ satisfies (1)-(3) as well. As a consequence, for any such $\{\Lambda_t\}$ (denoting by $A(t)$ and $B(t)$ the components of $B\setminus\Lambda_t$) there is a $t\in [0,1]$ with
\begin{equation}\label{swap}
\mbox{vol}(A(t)\cap C_1) = \frac{1}{2}\mbox{vol}(C_1)=\frac{2\pi}{3}.
\end{equation}

\subsection{Non-triviality of the sweepout}
In this section we prove
\begin{proposition}\label{wd}
For each $g\geq 1$ the width satisfies:
\begin{equation}\label{widthg}
 W_g>\pi=|D|.
\end{equation}
\end{proposition}
\noindent
The equation \eqref{widthg} expresses the non-triviality of the sweepout and is what permits the min-max method to work.

We need the following fundamental theorem about isoperimetric surfaces in the unit ball:  
\begin{theorem} (Almgren \cite{Al}, Bokowski-Sperner \cite{BS}, Ros \cite{R})\label{almgreniso}
The isoperimetric surfaces in $B$ are hyperplanes through the origin or spherical caps meeting $\partial B$ orthogonally.  Thus if a surface $\Sigma$ in $B$ with $\partial\Sigma\subset\partial B$ divides $B$ into two components with equal volume, then $|\Sigma|\geq\pi$, and equality holds if and only if $\Sigma$ is the intersection of $B$ with a hyperplane through the origin.
\end{theorem}

\noindent
\emph{Proof of Proposition \ref{wd}:}  \\
\indent
Assume toward a contradiction that $W_g=|D|=\pi$.  Thus by the definition of width, there is a sequence of sweepouts $\{\Sigma_t\}^i$ in the saturation $\Pi$ so that 
\begin{equation}
\sup_{t\in [0,1]}|\Sigma_t^j|  \leq  \pi+\epsilon_j,
\end{equation}
\noindent
for some sequence $\epsilon_j\rightarrow 0$.  

Since each surface $\{\Sigma_t\}^j$ divides $B$ into two components of equal volume, it follows that for each fixed $t$, $\{\Sigma_t\}^j$ is a minimizing sequence for the isoperimetric problem in $B$, and thus for large $j$ by Theorem \ref{almgreniso} each surface $\{\Sigma_t\}^j$ must be close as varifolds to some disk.  Since the sweepouts are $\mathbb{D}_{g+1}$ equivariant, it follows that this disk is precisely $D$.  In other words for some $\delta_i\rightarrow 0$, and $j$ large enough, we have 
\begin{equation}\label{w}
\sup_{t\in [0,1]} \mathbb{F}(\Sigma_t^j,D)\leq \delta_j,
\end{equation}
where $\mathbb{F}$ denotes the $\mathbb{F}$-metric on varifolds.

On the other hand, for any $j$, the sweepout $\{\Sigma_t\}^j$ must interchange the two components $C_1$ and $C_2$ of $B\setminus D$, which readily violates \eqref{w}.  

Let us give more details. By \eqref{swap}, since each sweepout $\{\Sigma_t\}^j$ interchanges $C_1$ and $C_2$, for each $j$, there is some $t_j\in [0,1]$ so that $\Sigma^i_{t_i}$ bounds two regions $R_1$, and $R_2$, where
\begin{equation}
\mbox{vol}(R_1\cap C_1) = \frac{1}{2}\mbox{vol}(C_1)=\frac{2}{3}\pi
\end{equation}
and 
\begin{equation}\label{www}
\mbox{vol}(R_2\cap C_1) = \frac{1}{2}\mbox{vol}(C_1)=\frac{2}{3}\pi.
\end{equation}
Choose $\epsilon>0$ and consider the cap
\begin{equation}
S_\epsilon := (C_1\setminus T_\epsilon(D))\cap B_{1-\epsilon}(0),
\end{equation}
where $B_{1-\epsilon}(0)$ denotes the ball of radius $1-\epsilon$ about the origin, and $T_\epsilon(D)$ is the $\epsilon$-tubular neighborhood about $D$ in $\mathbb{R}^3$.
The boundary of $S_\epsilon$ is contained in the sets
\begin{equation}
\partial S^1_\epsilon=\{(x,y,z)\in\mathbb{R}^3\;|\; z=\epsilon\}\cap B
\end{equation}
together with
\begin{equation}
\partial S^2_\epsilon=\{(x,y,z)\in\mathbb{R}^3\;|\; z\geq\epsilon\}\cap\partial B_{1-\epsilon}(0).
\end{equation}
\noindent
In light of \eqref{w}, for $t_i$ chosen above and $j$ large enough and any $\epsilon>0$ and any $\eta>0$
\begin{equation}\label{area11}
|\Sigma^j_{t_i}\cap S_\epsilon| < \eta.
\end{equation}
Thus by the coarea formula
\begin{equation}\label{co}
\int_\epsilon^{2\epsilon}\mathcal{H}^1(\Sigma^j_{t_i}\cap\partial S^1_\sigma)d\sigma \leq C|\Sigma^j_{t_i}\cap S_\epsilon| < C\eta.
\end{equation}
By \eqref{co} and Sard's lemma we can find $\sigma\in [\epsilon,2\epsilon]$ so that 
\begin{equation}\label{length2}
\mathcal{H}^1(\Sigma^j_{t_i}\cap\partial S^1_\sigma)\leq \frac{2C\eta}{\epsilon}.
\end{equation}
and
\begin{equation}
\Sigma^i_{t_i}\cap \partial S^1_\sigma
\end{equation}
consists of several closed circles or half-circles.  By the coarea formula again we can choose $\sigma$ so that
in addition $\partial S^2_\sigma$ consists of several closed circles or half-circles and
 \begin{equation}\label{area22}
\mathcal{H}^1(\Sigma^j_{t_j}\cap\partial S^2_\sigma)\leq \frac{2C\eta}{\epsilon}.
\end{equation}

By the isoperimetric inequality for the plane and $\partial B_{1-\sigma}$, \eqref{length2} and \eqref{area22} the disks $D_i$ in $\partial S_\sigma$ bounded by the collection of circles $\Sigma^j_{t_j}\cap \partial S_\sigma$ have total area:
\begin{equation}\label{area12}
\sum_i |D_i| \leq C'(\eta/\epsilon)^2.
\end{equation}

Note the decomposition
\begin{equation}\label{decomp}
S_\sigma = (\cup A_i)\cup X.
\end{equation}

In \eqref{decomp}, $A_i$ are the components of the interior of the closed surface obtained by capping off a component of $\Sigma^j_{t_j}$ in $\partial S_\sigma$ with one of the disks $D_i$ (i.e. ``filigree").  The set $X$ is defined to be the complement of the $A_i$.  Note that the sets $A_i$ are not necessarily pairwise disjoint as they may be nested.  Nonetheless, by the isoperimetric inequality in $\mathbb{R}^3$, \eqref{area11} and \eqref{area12} it follows that 
\begin{equation}\label{small}
\sum_i \mbox{vol}(A_i)\leq C ((\eta/\epsilon)^2 + C\eta)^{3/2}.
\end{equation}
Since $\sigma < \epsilon$ we can estimate the volume of the shell region
\begin{equation}\label{cyl}
\mbox{vol}(C_1\setminus S_{\sigma})  < C\epsilon.
\end{equation}
The set $X$ is connected and so is contained in either $R_1$ or $R_2$.  Suppose without loss of generality that $X\subset R_1$.  It follows that $R_2\cap C_1$ is composed of some of the filigree sets $A_i$ together with some parts of $C_1\setminus S_\sigma$.  Thus combining \eqref{cyl} with \eqref{small} we obtain:
\begin{equation}\label{howbigb}
\mbox{vol}(R_2\cap C_1)\leq  C ((\eta/\epsilon)^2 + C\eta)^{3/2} + C\epsilon.
\end{equation}
If first $\epsilon$ and then $\eta$ are chosen small enough, \eqref{howbigb} contradicts \eqref{www}.

\qed
\begin{remark}
As an alternative approach to proving Proposition \ref{wd} one could try to argue as follows. If $W_g=\pi$, then after a pull-tight procedure, each slice in a tightened sweepout $\Gamma_t$ (with maximal area very close to $\pi$) is close to a stationary varifold with area $\pi$.  Since each surface passes through the origin, by the monotonicity formula one obtains that all surfaces in $\Gamma_t$ are close to a disk.  The argument then proceeds as in Proposition \ref{wd}. This argument is less elementary since to use the monotonicity formula it seems one needs the integrality of limits of min-max sequences, which already uses the almost minimizing property.
\end{remark}
\subsection{Completion of the proof of Theorem \ref{main}}
Since $W_g>\pi$ by Proposition \ref{wd}, \eqref{isnontrivial} is satisfied, and thus we may apply Theorem \ref{eqminmax} to obtain a $\mathbb{D}_{g+1}$-equivariant free boundary minimal surface $\Sigma_g$ in $B$.  It remains to determine the topological type of $\Sigma_g$. 

In the following, we will first enumerate the possible neck-pinches and collapse of topology that may occur and the topological type of possible min-max limits.  With more work one can show that the neck-pinches we enumerate are the only possible ones.  The only fact we shall use in the proof of Theorem \ref{main} is that the genus of the min-max limit obtained is either $g$ or $0$, which follows just from the equivariance.  The reader may thus wish to skip to \eqref{islimit} as the intervening paragraphs below are not logically necessarily.  

Recall from Remark \ref{neckpinches} that there are three possible types of surgeries:  $\mathbb{Z}_{g+1}$ neckpinches occurring by removing an annulus centered around an arc of the singular set, surgeries removing a ``half-annulus" at the boundary $\partial B$, and \emph{ordinary} neckpinches performed in the interior of $B$ that occur in regions disjoint from $\mathcal{S}$.  In the following, we will describe the possible compressions.  

  Let us describe the potential $\mathbb{Z}_{g+1}$-neckpinch.  We can describe this compression on the initial sweepout surfaces $\tilde{\Sigma}_t$.  In the northern hemisphere, $\tilde{\Sigma}_t$ coincides with $C_t$.  Consider the circle $\gamma$ on $C_t$ that is homotopically nontrivial obtained by intersecting $C_t$ with a plane of constant positive $z$-value.  The circle $\gamma$ encircles the $z$-axis (which has isotropy $\mathbb{Z}_{g+1}$) and it is possible to compress along this circle.  By equivariance, if this circle is compressed, so must be a corresponding circle in the southern hemisphere.  After this surgery on $\tilde{\Sigma}_t$, one obtains two disks (one in each hemisphere) as well as one surface of genus $g$ which has one boundary circle.  Only this latter surface with genus $g$ can contribute to the min-max limit since the min-max limit contains the origin in its support.    Thus if this $\mathbb{Z}_{g+1}$ neckpinch occurs, one obtains a surface with one boundary component and genus $g$.  

We now will argue that there are no possible $\mathbb{Z}_2$-compressions or ordinary neckpinches that can bring down the genus of $\tilde{\Sigma}_t$.  To see this, consider any simple closed curve $\gamma$ in the interior of $B$ representing a non-trivial homology class of $\tilde{\Sigma}_t$.  Such a curve is contained in the tiny tubular neighborhood of $C_t\cap D$ on which $C_t$ and $D$ are desingularized to produce $\tilde{\Sigma}_t$.
But such a curve has non-zero intersection number with one of the lines $L_i$.  It follows that no such compression can occur since $\mathbb{Z}_2$-neckpinches as well as ordinary neckpinches by definition are obtained by removing an annulus that is disjoint from the singular set of the group action.   One way to understand this phenomenon is to consider Scherk's singly periodic surface in $\mathbb{R}^3$ that is the desingularization of two orthogonal planes and has dihedral symmetry.  There are two ways to ``snap the necks" but neither of the resulting configurations still has dihedral symmetry.  

One can still collapse the genus of $\tilde{\Sigma}_g$ by ``collapse of topology" as in Remark \ref{throwing}.  In this case, one can use an isotopy to press all of the genus of $\tilde{\Sigma}_t$ to the origin in a $\mathbb{D}_{g+1}$-equivariant fashion, in which case $\Gamma$ is a free boundary disk with odd multiplicity (by Theorem \ref{eqminmax}iii).  Since $W_g<|D|+|C|< 3|D|$, if follows that the multiplicity is $1$.  By a theorem of Nitsche \cite{N}, this disk must be flat, violating Proposition \ref{wd} as $W_g>|D|$.  Thus the throwing away of topology cannot occur.

Let us now consider the compressions that occur along ``half-annuli" at the boundary $\partial B$.  Let us describe the arcs along which the compression can occur in the model $\tilde{\Sigma}_t$ surface.  Recall that in producing $\tilde{\Sigma}_t$, along the arc $A_1$ of $C_t\cap D$, the part of $C_t$ contained in the northern hemisphere is being connected with the component of the disk $D\setminus (C_t\cap D)$ that touches $\partial B$.  Thus there is an arc $\alpha$ contained on $\tilde{\Sigma}_t$ beginning at the boundary $\partial B\cap C_t$ and ending at $\partial B\cap D$.  It is clear $\alpha$ bounds a ``half-disk" and such a surgery is admissible.  By the equivariance, if such a surgery is performed, there are $2(g+1)$ copies of it that must be performed concurrently.  The effect of this surgery is to bring the genus of $\tilde{\Sigma}_t$ down to zero while the number of ends becomes one (though this end is rather checkered).   After performing this surgery, one can see that no further ones are possible. 

In total, we have shown either i) $\Sigma_g$ has genus $g$ and three ends, or else ii) $\Sigma_g$ has genus $g$ and one end, or else iii) genus zero and one end.   Let us now show that when $g$ is large, the last two cases cannot occur.  

To achieve this, we prove \eqref{limit}, i.e., that 
\begin{equation}\label{islimit}
\Sigma_g\rightarrow D\cup C,
\end{equation}
\noindent
from which it follows that when $g$ is large, $\Sigma_g$ has genus $g$ and three ends.  

Consider any subsequence (not relabelled) of $\Sigma_g$.  We will show that $\Sigma_g$ has a subsequence converging to $D\cup C$, which implies \eqref{islimit}.  Let $\Sigma_\infty$ denote the (free boundary) stationary integral varifold that is a subsequential limit of $\Sigma_g$ as $g\rightarrow\infty$.  Since $\Sigma_g$ is invariant under $\mathbb{D}_{g+1}$, it follows that the support of $\Sigma_\infty$ is invariant under rotations about the $z$ axis.   

Since $\Sigma_g$ contains the lines $\{L_i\}_{i=0}^{g}$ that become denser and denser as $g\rightarrow\infty$, it follows that the disk $D$ is in the support of $\Sigma_\infty$.   By the monotonicity formula, it also follows that the support of $\Sigma_\infty$ is connected.  

Let us first consider the blowup set $\mathcal{B}$ for the curvature of $\Sigma_g$:
\begin{equation}
\mathcal{B} = \{x\in B\;|\; \inf_{r>0}\liminf_{g\rightarrow\infty} \int_{B_r(x)}|A|_{\Sigma_g}^2d\mu\geq\epsilon_0\}
\end{equation}
The constant $\epsilon_0>0$ is chosen so that by the $\epsilon$-regularity theorem of Choi-Schoen \cite{CS}, if $x\in B\setminus\mathcal{B}$, then some subsequence of $\Sigma_g$ satisfies uniform curvature estimates in a neighborhood of $x$.

%$\mathcal{B}$ is non-empty since the genus of $\Sigma_g$ is $g$. 
 We claim:
\\
\\
\emph{The set $\mathcal{B}\cap\mbox{int}(B)$ consists of a single circle $S$ centered about $0$ in $D$ of some radius $r\in (0,1)$.}  
\\
\\ \noindent
Let us first prove this claim.  By the equivariance of $\Sigma_g$ which increases as $g\rightarrow\infty$ it follows that $\mathcal{B}$ is the union of (potentially infinitely many) round circles about the $z$-axis.

Let us first show that $\mathcal{B}\cap\mbox{int}(B)$ contains a circle in the disk $D$.  Suppose not.  Then for any sub-disk $D'$ of $D$ (passing to a subsequence in $g$) there exists a $r>0$ so that
\begin{equation}\label{lowcurv}
\sup_{x\in T_r(D')}|A|^2_{\Sigma_g}\leq C
\end{equation}
\noindent
where $T_r(D')$ denotes the $r$-tubular neighborhood about $D'$.  By \eqref{lowcurv}, it follows that $\Sigma_g\rightarrow\Sigma_\infty$ smoothly on compact subsets of $D$.  By the free boundary condition and the unique continuation property for minimal surfaces (Lemma 5 in \cite{MY}) and since $\Sigma_\infty$ contains $D$, it follows that $\Sigma_\infty$ restricted to $T_r(D)$ consists of $kD$ where $k$ is a positive integer.  By the equivariance, it follows that $k$ is odd.  Since $|\Sigma_g|<|D|+|C|< 3|D|$, it follows that $k=1$.   As $\Sigma_\infty$ restricted to $T_r(D)$ is precisely $D$, the connectness of $\Sigma_\infty$ implies that $\Sigma_\infty=D$. 

To rule this situation out, observe that if indeed $\Sigma_g\rightarrow D$, then for $g$ large enough, the area of $\Sigma_g$ approaches that of $D$.  By Proposition \ref{rigidity}, this implies $\Sigma_g$ is itself $D$, which contradicts the fact that $|\Sigma_g|>|D|$.   Thus we have a contradiction, and it follows that $\mathcal{B}$ contains a circle $S$ in the interior of the disk $D$.  Note that circle may be trivial, i.e. have zero radius, and we will rule this situation out later.

We claim that for some $r_i(g)\rightarrow 0$, the genus of $\Sigma_g$ restricted to $T_{r_i}(S)$ is equal to $g$.  To see this, suppose that instead $T_{r}(S)$ contains no genus for $g$ large (by the equivariance, the genus $g$ is either all contained in $T_r(S)$ or zero in $T_r(S)$).   Then by Ilmanen's integrated Gauss-Bonnet argument (Lemma 1 in Lecture 3 in \cite{I}), one obtains for some $C<\infty$ 
\begin{equation}\label{hm}
\sup_{g\in\mathbb{N^+}}\int_{\Sigma_g\cap T_{r/2}(S)}|A|^2 \leq C.
\end{equation} 
However, since $S\subset\mathcal{B}$, then by definition of $\mathcal{B}$ and the rotational invariance we obtain
\begin{equation}\label{hm}
\sup_{g\in\mathbb{N^+}}\int_{\Sigma_g\cap T_{r/2}(S)}|A|^2 = \infty.
\end{equation}
This is a contradiction.

Since we  have just proved that the genus is collapsing into $S$, it follows that 
\begin{equation}\label{onecircle}
\mathcal{B}\cap\mbox{int}(B)=S.
\end{equation}

Indeed, let $S_2$ be another circle contained in $(\mathcal{B}\cap\mbox{int}(B))\setminus S$.  Since the genus of $\Sigma_g$ in a neighborhood of $S_2$ is $0$ (as the genus is contained in smaller and smaller neighborhoods about $S$), we can again apply Ilmanen's integrated Gauss-Bonnet argument \cite{I}, to obtain for some $\epsilon>0$ sufficiently small and $C_\epsilon<\infty$  
\begin{equation}\label{hm}
\sup_{g\in\mathbb{N^+}}\int_{\Sigma_g\cap T_{\epsilon}(S_2)}|A|^2 \leq C_\epsilon.
\end{equation}

But since we have assumed $S_2\subset\mathcal{B}$, then by definition of $\mathcal{B}$ and the rotational invariance we obtain
\begin{equation}\label{hm2}
\sup_{g\in\mathbb{N^+}}\int_{\Sigma_g\cap T_{\epsilon}(S_2)}|A|^2 = \infty.
\end{equation}
\noindent
The identity \eqref{hm2} contradicts \eqref{hm} and thus in fact \eqref{onecircle} holds.

To prove the claim, it now suffices to show the circle $S$ is non-trivial, i.e. its radius $r$ is not zero. If instead $r=0$, then $\Sigma_g$ converge to a multiple of the disk $D$.  One can see this as follows.  In this case, the curvature of $\Sigma_g$ is bounded away from the origin, so that $\Sigma_\infty\setminus\{0\}$ is a smooth free boundary minimal surface (maybe with multiplicity).  Since $\Sigma_\infty$ contains the disk, $\Sigma_\infty$ is either $D$ or some multiple $kD$. Any other component $C_\infty$ would be a smooth rotationally symmetric disk containing the origin in the northern hemisphere, together with its mirror image in the south.  But this violates the maximum principle unless both of these components are $D$.  So in fact $\Sigma_\infty=kD$ for some integer $k$.  We have already ruled out the case $k=1$.  From the equivariance of $\Sigma_\infty$, it follows that $k$ is odd.  But because $|\Sigma_g|<|C|+|D|< 3|D|$ we obtain $|\Sigma_\infty| < 3|D|$ and thus $k<3$.  Thus the case $r=0$ is ruled out and $r$ is instead some value in $(0,1)$.

In sum we obtain that $\Sigma_\infty$ is a rotationally symmetric stationary varifold in $B$ containing $D$ that is singular only at the circle $A(r)$ for some $0<r<1$.  Since $\Sigma_\infty\neq D$,  it follows that $\Sigma_\infty$ contains in the northern hemisphere (and by symmetry in the southern too) an additional smooth rotationally symmetric minimal component $\Sigma'_\infty$ with boundary $A(r)$. This component either hits the boundary $\partial B$ or stays in the interior of $B$.  If it stays in the interior, then by the convex hull property of minimal surfaces, $\Sigma'_\infty$ is contained in $D$, and thus $\Sigma_\infty$ has differing integer multiplicities in the disk $D'$ bounded by $A(r)$ and its complement $D\setminus D'$ .  This violates the Constancy Theorem which states that a stationary varifold supported on a smooth surface is an integer multiple of the surface.  If instead $\Sigma'_\infty$ reaches $\partial B$, then because the critical catenoid is the unique rotationally symmetric free boundary surface aside from the disk, it follows that $\Sigma_\infty=D\cup C$.  

This completes the proof of Theorem \ref{main}.
\qed 

\begin{remark}
Assuming one knew that $\Sigma_g$ had three boundary components, then instead of using Proposition \ref{rigidity} to rule out the situation where $\Sigma_g\rightarrow D$ one could alternatively apply the following argument.  By Theorem 5.4 in Fraser-Schoen \cite{FS3}, since $\Sigma_g$ is a free boundary minimal surface, one has 
\begin{equation}\label{fsgreat}
|\partial\Sigma_g|=2|\Sigma_g|.  
\end{equation}
Thus if 
\begin{equation}
|\Sigma_g|\rightarrow\pi,
\end{equation} then by \eqref{fsgreat}
\begin{equation}\label{boundary}
|\partial\Sigma_g|\rightarrow 2\pi.
\end{equation}  

By assumption $\partial\Sigma_g$ consists of three circles.  If $\Sigma_g\rightarrow D$, then either the middle circle converges to $\partial D$ and the top and bottom ones vanish at the north and south poles or else the three circles converge to $\partial D$ with multiplicity $3$.  The first case is ruled out by applying the monotonicity formula to $\Sigma_g$ at points near the north or south pole.  The second case is ruled out by \eqref{boundary} as  $|\partial\Sigma_g|$ would be converging to $6\pi$ in this case.
\end{remark}

\section{Variational construction of the Fraser-Schoen examples}\label{fraserschoen}
In this section, we prove Theorem \ref{fs} which we restate 

\begin{theorem}
For each $k\geq 2$, a free boundary minimal surface isotopic to $F_k$ can be constructed variationally through a one-parameter equivariant min-max procedure.  
\end{theorem}
\begin{proof}
Let $G=\mathbb{D}_{k}$.  We construct a $G$-invariant sweepout of $B$ as follows.   For $t\in (-1,1)$, denote
\begin{equation}
D_t = B\cap\{(x,y,z)\in\mathbb{R}^3\;|\; z = t\}.
\end{equation}
Note that
\begin{equation}\label{area}
|D_t| = \pi(1-t^2).
\end{equation}

For $t\in [0,1]$ consider the surfaces $\Sigma_t= D_t\cup D_{-t}$ (so that $\Sigma_0$ is the disk $D_0=D$ with multiplicity $2$).   Fix $k$ evenly spaced lines $L_i$ of longitude joining the north pole of $\partial B$ to the south pole.  Let $L_i(\epsilon)$ be the $\epsilon$-tubular neighborhood (in $\mathbb{R}^3$) of $L_i$ intersected with $B$ and let  $P_i(\epsilon)$ denote the part of the boundary of $L_i(\epsilon)$ contained in the interior of $B$.

Now for some function $f(t):[0,1]\rightarrow [0,\delta]$ consider the surfaces
\begin{equation}
\tilde{\Sigma}_i = \Sigma_t\cup\bigcup_i P_i(f(t))\setminus (\bigcup_i\Sigma_t\cap L_i(f(t)).
\end{equation}

In other words, we add to $\Sigma_t$ the half-tubes $P_i$ and remove the two half-disks in $\Sigma_t$ abutting the boundary of $B$.  By construction $\tilde{\Sigma}_t$ are $\mathbb{D}_k$-equivariant, have genus $0$ and $k$ ends as long as $f(t)>0$.

Choose $f(t)$ so that $f(t)\rightarrow 0$ as $t\rightarrow 1$.  In other words, as $t\rightarrow 1$, the necks dissappear and $\tilde{\Sigma}_t$ converges to the graph $\mathcal{G}_1$ consisting of the collection of arcs $L_i$.  Fix $\epsilon>0$ small and enforce for $f$ that $f(\epsilon)=0$ and also $f(x)>0$ for $x\in [\epsilon,1)$.  

Thus we have in light of \eqref{area} and the choice of $f$
\begin{equation}
|\tilde{\Sigma}_\epsilon|\leq 2|D|-2\pi \epsilon^2
\end{equation}
and in fact
\begin{equation}\label{goodarea}
\sup_{t\in [\epsilon,1]} |\tilde{\Sigma}_t| \leq 2|D| - C,
\end{equation}
\noindent
for some $C>0$.

For each $i$, let $p_i$ denote the point $L_i\cap D$.  Note that there is a retraction $R_t$ from $D\setminus\cup_i p_i$ onto the graph $\mathcal{G}_2$ consisting of $k$ equally spaced lines passing through the origin in $D$.   Thus applying the Catenoid Estimate \cite{KMN} and using \eqref{goodarea} one can
adjust the sweepout $\tilde{\Sigma}_t$ in the region $t\in [0,\epsilon]$ so that
\begin{equation}\label{areas}
\sup_{t\in [0,1]} |\tilde{\Sigma}_t| < 2|D|.
\end{equation}
and $\tilde{\Sigma}_t\rightarrow\mathcal{G}_2$ as $t\rightarrow 0$.

Applying Theorem \ref{eqminmax} to the family $\tilde{\Sigma}_t$ and the group $G$, we obtain a free boundary minimal surface $F_k$.  The only possible neckpinches would result in two disks, which would imply $F_k=2D$, contradicting \eqref{areas}.  Thus there are no neckpinches and one obtains free boundary minimal surfaces isotopic to those of \cite{FS}.
\end{proof}
\section{New free boundary minimal surfaces associated to the Platonic solids}\label{platonics}
In this section, we prove Theorem \ref{platonic}, which we restate:

\begin{theorem}
There exists a free boundary minimal surface in $B$  with octahedral symmetry of genus $0$ and $6$ ends, an example with tetrahedral symmetry of genus $0$ and $4$ ends, and an example of genus $0$ and $12$ ends of dodocahedral symmetry.
\end{theorem}
\begin{proof}
Let us produce the surface of genus $0$ and $6$ ends, as the others follow analogously.  Consider $O_{24}$  (the symmetry group of a cube) acting on $B$.   We can construct a $O_{24}$-invariant sweepout $\Sigma_t$ of $B$ as follows.   Let $\Sigma_0$ be the graph $\mathcal{G}$ consisting of the $x$-axis, $y$-axis, and $z$-axis restricted to $B$.  For $t$ small, set $\Sigma_t = \partial T_\epsilon(\mathcal{G})$.  One can extend this sweepout $O_{24}$-equivariantly to the rest of $B$ so that $\Sigma_1$ consists of the tessellation of $\partial B$ by six squares.  

Applying Theorem \ref{eqminmax} to the sweepout $\Sigma_t$ with $G=O_{24}$ one obtains an embedded connected free boundary minimal surface $\Gamma$.  It is easy to see that the only possible degeneration is into several disjoint disks.  But no disk is invariant under $O_{24}$.  It follows that $\Gamma$ is a free boundary minimal surface with six ends and genus zero, resembling a three-dimensional ``cross."
\end{proof}
\section{Proof of the Min-Max Theorem \ref{eqminmax}}\label{proof}
We will assume the reader is familiar with the min-max construction of Simon-Smith \cite{ss} and only focus on the changes necessary from the standard arguments.  For an exposition of the theory see for instance Colding-De Lellis \cite{cd}, or \cite{ketover} for the control on the genus of the limiting minimal surface.  In Section 2 of \cite{ketover2} is a detailed account of the changes needed in the equivariant setting.  De-Lellis-Ramic provide a detailed account of the min-max theory in the free boundary setting (the ``unconstrained problem") which includes a proof of the regularity of free boundary minimal surfaces at their free boundary.

Note that the ``boundary" considered here is not entirely ``free."  Consider for instance the situation of Theorem \ref{main} where the sweepout surfaces contain an axis of the singular set $\mathbb{Z}_2$ and so cannot move off these lines.  Thus the boundary is constrained to contain the points that are the intersections of these lines with $\partial B$.  But this is not really a problem because of the Schwarz reflection principle: if a varifold is stationary with respect to variations preserving a line in the surface, it is stationary with respect to all variations (cf. Lemma 3.8 in \cite{ketover2}).

\subsection{Existence of a free boundary equivariant stationary varifold $\mathcal{V}$}
\indent
Given a vector field $\chi$ defined on $\overline{B}$, denote by $\phi_\chi$ the family of isotopies obtained by integrating $\chi$.  Let us consider the family of vector fields $\chi$ so that the corresponding isotopy $\phi_\chi(t):\overline{B}\rightarrow\overline{B}$ satisfies
\begin{enumerate}
\item For all $t\in [0,1]$, $\phi_t$ preserves $\partial B$. 
\item $g\phi_t(x)=\phi_t(gx)$ for all $t$, $g\in G$ and $x\in\overline{B}$.
\end{enumerate}
Denote by $Is_{FB}^G$ the set of vector fields $\chi$ so that $\phi_\chi$ satisfies items (1) and (2).  The set $Is_{FB}^G$ consists of the \emph{free boundary equivariant vector fields}.  A varifold $\mathcal{W}\subset B$ is called equivariant if $g_{\#}(\mathcal{W})=\mathcal{W}$ for all $g\in G$.  

An equivariant varifold $\mathcal{W}$ is called a \emph{free boundary equivariantly stationary varifold} if $\delta_V(\mathcal{W})=0$ for all vector fields $V\in Is_{FB}^G$.   Note that such $V$ are tangent to $\partial B$ by definition.  

As observed in \cite{ketover2}, the set of vector fields $Is_{FB}^G$ is a convex vector space.  Thus the ``pull-tight" procedure (Proposition 3.2 in \cite{DR}) applies to produce a free boundary equivariantly stationary varifold $\mathcal{V}$.  By Lemma 3.8 in \cite{ketover2}, a free boundary equivariantly stationary varifold is in fact stationary.  Thus we produce an equivariant, free boundary stationary varifold $\mathcal{V}$ with mass equal to $W^G_\Pi$.

\subsection{Regularity of $\mathcal{V}$}
As in Section 4.2 in \cite{ketover2} and Proposition 4.3 in \cite{DR}, one can find a min-max sequence $\Sigma_j$ that is almost minimizing in annuli small enough.  Precisely, for each $x\in\overline{B}$, there is a radius $r(x)>0$ so that for any annulus $An$ about $x$ of outer radius at most $r(x)$, $\Sigma_j$ is $1/j$ $G$-almost minimizing in $An$.  In other words, $\Sigma_j$ is almost minimizing but only among variations through $G$-equivariant deformations.  

We must construct smooth replacements for $\mathcal{V}$ in such annuli.  There are several cases: 
\begin{enumerate}
\item $x\notin\mathcal{S}$ 
\item $x\in\mathcal{S}$ and $\Sigma_t$ do not contain the arc $\mathcal{A}$ containing $x$ 
\item $x\in\mathcal{S}$ and $\Sigma_t$ \emph{does} contain the arc $\mathcal{A}$ containing $x$ 
\end{enumerate}

We can further subdivide case (1) (2) and (3) into subcase a) if $x$ is in the interior of $B$ and b) if $x$ is in the boundary of $B$.

Case (1ab) and (2a) follow from previous work on free boundary problems (see for instance \cite{GJ}) and \cite{ketover}. Let us first consider (3ab).

Let us first recall the following definitions (Section 4 in \cite{ketover2}):

\begin{definition}\normalfont
Let  $\Sigma$ be a smooth $G$-equivariant surface contained in a $G$-ball.  Choose a normal vector field $n$ on $\Sigma$.  Let us call a smooth function $\phi$ defined on $\Sigma$ an {\it equivariant deformation} if for all $t$ small enough, the following set is $G$-equivariant:
\begin{equation}\label{eqdef}
\Sigma_{t\phi} = \{\exp_p(n(p)t\phi(p)) \;|\; p\in \Sigma\}.
\end{equation}
In other words, $\phi$ is an equivariant deformation if moving normally to $\Sigma$ according to $\phi$ gives rise to $G$-equivariant surfaces.  
\end{definition}

\begin{definition}
A $G$-equivariant surface $\Sigma$ is \emph{$G$-stable} if it is stable among equivariant deformations, i.e., 
\begin{equation}
\frac{d^2}{dt^2}\Big|_{t=0} \mathcal{H}^2(\Sigma+nt\phi) \geq 0,
\end{equation}
for all equivariant deformations $\phi$.
\end{definition}

%Suppose $\Sigma$ is a surface in $\mathbb{R}^3$ containing a line $\mathcal{L}$ and suppose the surface is invariant %about $180^o$ rotations about $\mathcal{L}$.  For instance, Scherk's singly periodic surface contains many such lines.  

Given a point $p$ in case (3a), one can minimize among $G$-equivariant $1/j$-isotopies in a fixed annulus $An$ based about $p$ to produce a $\mathbb{Z}_2$-stable surface $V_j$ which is a replacement for $\Sigma_j$ in $An$.  Assume for the moment that $V_j$ is smooth.

If $V_j$ intersects $\mathcal{S}$ transversally (as in case (2)), it was proved in Section 4 in \cite{ketover2} that $G$-stability implies stability.  
In case (3a) however,  $V_j$ contains an arc $\mathcal{A}\subset\mathcal{S}$.  In this setting, $\mathbb{Z}_2$-stability may not be equivalent to stability.   It is easy to see that $V_j$ is stable among odd deformations, but it may not be stable among even deformations.  Given any fundamental domain $\mathcal{F}$ of the $\mathbb{Z}_2$ action, any compactly supported variation $\chi$ supported in $\mathcal{F}$ extends by $\chi(\tau(x)):=-\chi(x)$ to give an odd variation of $An$.  Thus $V_j$ is stable in any such $\mathcal{F}$. By curvature estimates for stable surface \cite{schoen}, this implies that 
\begin{equation}\label{schoenaxis}
|A|^2(x)\leq C\mbox{dist}(x,(\mathcal{A}\cap An)\cup\partial An)^{-2}.
\end{equation}
In other words, the curvature of $V_j$ is bounded in the interior of $An$ but may blow up as one approaches the axis $\mathcal{A}$.

On the other hand, since $\Sigma_j$ has bounded genus and area, it follows that $V_j$ have bounded genus and area (see for instance Proposition 4.7 in \cite{ketover}).  By the integrated Gauss-Bonnet argument of Ilmanen (Lemma 1 in Lecture 3 in  \cite{I}), in any proper subannulus $An'\subset An$, there holds
\begin{equation}\label{bounded}
\sup_j \int_{V_j\cap An'} |A|_{V_j}^2\leq C(An'), 
\end{equation}
where $C(An')$ is a constant depending on $An'$ which blows up as $An'$ approaches $An$.  Note that one cannot expect any curvature bound up to the boundary of $An$.  Consider for instance the Fraser-Schoen free boundary minimal surfaces $F_k$ in $B$.  They have genus $0$ and bounded areas, but do not converge smoothly to $2D$ over $\partial B$.

Using \eqref{bounded}, in any such $An'$, $V_j$ has a convergent subsequence by classical results due to Choi-Schoen \cite{CS}.  Taking a sequence of annuli $An'\rightarrow An$ and a diagonal argument, one can produce a subsequence of $V_j$ converging to a smooth $\mathbb{Z}_2$-stable minimal surface $V_\infty$ in compact subsets of $An$.  Note that the convergence may not be smooth in any given subannulus, $An'$.  There may be finitely many points where the convergence fails to be smooth.  Note also that the convergence $V_j\rightarrow V_\infty$ is smooth in the interior of $An$ away from the axis $\mathcal{A}$ by \eqref{schoenaxis}. 

Thus we have replaced $\mathcal{V}$ in $An$ by a smooth minimal surface $V_\infty$.  Moreover, the replacement is stationary over $\partial(An)$ as in Proposition 7.5 in \cite{cd} and thus gives a replacement for $\mathcal{V}$.  

In Section 6 of \cite{cd}, the fact that $V$ has local smooth replacements is used to prove the regularity of $\mathcal{V}$.  One can peruse the proof of Theorem 6.3 in \cite{cd} to see that the curvature estimates of Schoen's are not needed at all.  Colding-De Lellis explicitly point this out in Section 2.5: ``In fact what we will use is not the actual curvature estimate, rather it is the following consequence of it:  a sequence of stable minimal surfaces has a convergent subequence."   We  replace this compactness theorem of Schoen \cite{schoen} with the fact that $\mathbb{Z}_2$-stable surfaces with bounded area and genus subconverge smoothly away from finitely many points concentrating along the axis $\mathcal{A}$.  This then implies regularity of $\mathcal{V}$.

It remains to prove that one can minimize among $G$-equivariant $1/j$-isotopies in $An$ to produce the smooth $\mathbb{Z}_2$-stable replacements $V_j$.

As observed in Section 4.3 in \cite{ketover2}, by a Squeezing Lemma, to prove regularity of the replacements $V_j$ it is enough to prove the following Proposition \ref{minimizing}, which states that one can minimize area restricting to equivariant isotopies.  The only difference here from Section 4.3 in \cite{ketover2} is that one of the curves in the boundary is not acted upon freely by $\mathbb{Z}_2$.  Some of the proof is identical to Proposition 4.14 in \cite{ketover2}.   The added difficulty is that here it is does not follow in the same way as in \cite{ketover2} that area minimizing disks are equivariant (though I do not know an explicit example where this fails).

\begin{proposition}\label{minimizing}
Suppose $\mathbb{Z}_2$ acts on a $3$-ball $B$ as a rotation of $180^o$ about the line $\mathcal{S}$.  Let $\{\gamma_i\}_{i=1}^k$ be a collection of Jordan curves in $\partial B$ bounding a $\mathbb{Z}_2$-equivariant surface $\Sigma\subset B$ of genus $0$ and so that $\mathbb{Z}_2$ acts freely on the curves $\{\gamma_i\}_{i=2}^k$ but non-freely on $\gamma_1$ (and thus $\gamma_1$ contains the two points of $\mathcal{S}\cap\partial B$).  Consider a minimizing sequence $\Sigma_i$ for area among surfaces that are contained in $B$ and $\mathbb{Z}_2$-isotopic to $\Sigma$. Then after passing to a subsequence (not relabeled) $\Sigma_i$ converges with multiplicity $1$ to a smooth embedded $\mathbb{Z}_2$-equivariant minimal surface $V$ with boundary $\{\gamma_i\}_{i=1}^k$ and genus $0$.
\end{proposition}

\begin{proof}
Regularity of $V$ away from the axis $\mathcal{S}$ follows from the replacement theory \cite{pitts} and \cite{msy}.  It remains to prove that $V$ extends smoothly over $\mathcal{S}$. To that end, we can fix a small ball $N$ centered about $\mathcal{S}$. As in Proposition 4.14 in \cite{ketover2}, Steps 1) and 2), we can perform finitely many neckpinches so that the surgered sequence $\Sigma_i$ still converges to $V$ and moreover, we can replace $\Sigma_i$ so that it consists of disks $\cup_k D^j_k$ inside $N$.   Let $D_{k'}^j$ denote the unique disk that in $N$ with containing the singular axis.  We then can replace each disk in the collection of disks $(\cup_k D^j_k)\setminus D_{k'}^j$ with the area minimizing disk $\cup_{k\neq k'} A^j_k$ in $N$ with the same boundary in $\partial N$.  By the Meeks-Yau cut-and-paste argument \cite{MY} the disks in the collection $\cup_{k\neq k'} A^j_k$ are pairwise disjoint (since their boundaries are) and are moreover minimizers among all $\mathbb{Z}_2$ isotopies.  The disk $D^j_{k'}$ is problematic since it is not clear that the area-minimizer is $\mathbb{Z}_2$-equivariant.  If it is, then we can replace $D^j_{k'}$ with the area-minimizer, which again by Meeks-Yau is disjoint from the disks $\cup_{k\neq k'} A^j_k$.  In that case, taking $j\rightarrow\infty$ we produce a replacement for the varifold $V$ in $N$, which implies regularity as in \cite{pitts} \cite{cd}.

Suppose instead the area minimizer $S$ for disks with boundary $\gamma_1$ in $N$ is not $\mathbb{Z}_2$-equivariant.  Then consider $\tau(S)$, where $\tau$ generates the $\mathbb{Z}_2$ action.  By the Meeks-Yau argument, $\tau(S)$ and $S$ are disjoint except for their boundaries which coincide.  Since $S$ and $\tau(S)$ are area minimizing, it follows by Meeks-Yau again that $\cup_{k\neq k'} A^j_k$ is disjoint from $\tau(S)$ and $S$ and thus avoids entirely the $3$-ball $R$ in $N$ bounded between $\tau(S)$ and $S$.  By Lemma \ref{eqmin} below we can minimize area for $D^j_{k'}$ in the region $R$ among $\mathbb{Z}_2$-equivariant isotopies to produce a $\mathbb{Z}_2$-equivariant minimal disk $A^j_{k'}$ in $R$ (though potentially not stable among all variations).  By Schoen-Simon \cite{schoensimon} it follows that simply connected minimal surfaces satisfy curvature estimates.  Thus $A^j_{k'}$ still has a convergent subsequence as $j\rightarrow\infty$.  The limit in $j$ of this disk together with the limits of $\cup_{k\neq k'} A^j_k$ gives a smooth replacement for $V$ in $N$, and thus establishes regularity. 
\end{proof}
\begin{lemma}\label{eqmin}
Let $N$ be a domain in $\mathbb{R}^3$ with boundary consisting of two minimal disks $D_1$, $D_2$.  Assume $N$ is invariant under a rotation $\tau$ of $180^o$ that interchanges $D_1$ and $D_2$.  Let $\mathcal{S}$ denote the singular set of the $\mathbb{Z}_2$ action which intersects $\partial D_1=\partial D_2$ in two points and is contained in $N$.  Let $\gamma$ be the $\mathbb{Z}_2$-equivariant closed curve given by $\partial D_1=\partial D_2$ bounding a disk $\Sigma$ in $N$ that contains $\mathcal{S}$.  Then one can minimize area for $\Sigma$ restricting to $\mathbb{Z}_2$-equivariant isotopies in $N$ to produce a $\mathbb{Z}_2$-stable minimal disk $\Sigma_\infty$ containing $\mathcal{S}$ with boundary $\gamma$.  
\end{lemma}
\begin{proof}
It follows as before that the minimizing sequence among $\mathbb{Z}_2$ isotopies converges to a smooth minimal disk away from $\mathcal{S}$. Since each half of $\Sigma_\infty$ contains the line $\mathcal{S}$, by Schwarz reflection, $\Sigma_\infty$ is a smooth surface, embedded away from $\mathcal{S}$, but potentially containing points of self-intersection at $\mathcal{S}$.  Thus we need only show that any tangent cone of $\Sigma_\infty$ at $\mathcal{S}$ consists of a plane with some multiplicity.  

The proof of this fact is essentially contained in Theorem 3 in \cite{almgrensimon}  and Section 8 in \cite{dp} so we merely sketch the argument.  Consider a sequence of dilations $\lambda_j\rightarrow\infty$ and the sequence of rescaled surfaces $\tilde{\Sigma}_j:=\lambda_i(\Sigma_j-x)$ approaching a tangent cone $C$ at $x$.  Let us restrict attention to the sequence $\tilde{\Sigma}_j$ contained in the unit ball $\mathcal{B}_1$ in $\mathbb{R}^3$. The cone $C$ in $\mathcal{B}_1$ consists of several half disks $\cup_k P_k$ meeting along the $z$-axis.  This follows as the surface $\Sigma_\infty$ has bounded genus and is smooth away from the axis $\mathcal{S}$ and thus a tangent cone at $\mathcal{S}$ cannot acquire any singular points beyond $\mathcal{S}$ itself.  By a desingularization procedure (Section 8.3 in \cite{dp}), one can assume $\tilde{\Sigma}_j\cap\mathcal{B}_1$ consists of several disks $\cup_i D^j_i$, and the area of each disk is very close to a minimal area disk with the same boundary values.  One of these disks $D^j_1$ is the one containing $\mathcal{S}\cap\mathcal{B}_1$, and by the $\mathbb{Z}_2$-equivariance this disk converges to a union of two half-disks, say $P_1\cup P_2$, where $P_2$ is the half disk obtained by rotating $P_1$ by $180^o$.  The other disks $\cup_{i\geq 2} D^j_i$ either have no limit in $C$, or secondly converge to the flat disk $P_1\cup P_2$ or thirdly can converge to the union of several of the half-disks $P_i$ that are contained in one of the hemispheres $\mathcal{B}\setminus(P_1\cup P_2)$.    This third possibility cannot occur because such disks would be ``folding" along a line and thus be very far in area from the infimal area of disks with their boundary values.  See Section 8.5 in \cite{dp} for a demonstration of this, or the proof of Theorem 3 in \cite{almgrensimon}.
\end{proof}

As for cases (2b), and (3b), the changes necessary are minor.  Fix $x\in\partial B$ and an annulus $An$ centered about $x$ for case (2b) or (3b).  Then one can minimize appropriate $1/j$-isotopies (allowing the boundary of $\Sigma_j$ in $\partial B$ to move) as in Section 9 in \cite{DR} to obtain a smooth free boundary minimal surface $V_j$ in $An$.  In case (2b) one obtains a $\mathbb{Z}_n$-stable free boundary minimal surface $V_j$ in $An$.  The surface produced is in fact stable among all variations for the free boundary problem by Proposition 4.6 in \cite{ketover2}.   Thus one can apply the curvature estimates for such surfaces (Theorem 7.3 in \cite{DR}) to obtain a smooth replacement in $An$. In case (3b) one obtains a $\mathbb{Z}_2$-stable free boundary minimal surface. Since $\mathbb{Z}_2$-stability may not imply stability in this setting, one can appeal to the arguments of case (3a) to obtain nevertheless a convergent subsequence away from finitely many points in $An$.

\subsection{Completion of proof of Theorem \ref{eqminmax}}
Since the regularity of $\mathcal{V}$ has been established, the other claims Theorem \ref{eqminmax}ii, iii, iv and v. follow easily.   The proofs of iv. and v. are identical to the arguments in Section 5 of \cite{ketover2}.  Let us show iii.

In this case, the min-max sequence $\Sigma_i$ contains the singular axis $\mathcal{A}$ with $\mathbb{Z}_2$ isotropy. 
Let us first show that $\Gamma$ also contains $\mathcal{A}$.  Fix a point $x\in\mathcal{A}$ and a ball $B$ about $x$ so that $B\cap\mathcal{S}=\mathcal{A}'$ and $\Sigma_i$ intersects $\partial B$ transversally.  For each $i$, there is a distinguished circle $C_i$ contained in $\partial B$ that contains the two points $N$ and $S$ of intersection of $\mathcal{A}'$ with $\partial B$.   The circle $C_i$ is comprised of two arcs $A_i$ and $B_i$, each starting at $N$ and ending at $S$ and being interchanged by the $\mathbb{Z}_2$ action $\tau$.  Consider the piece $P_i$ of the surface $\Sigma_i$ in $B$ bounded by the closed piecewise smooth curve $\mathcal{A}'\cup A_i$.  Since this curve is evidently bounded away from zero in the flat topology as $i\rightarrow\infty$, it cannot happen that the area of $P_i$ tends to zero as $i\rightarrow\infty$.  Thus the varifold limit of $P_i$ is contained in the min-max minimal surface $\Gamma$.  Since each $P_i$ contains $\mathcal{A}$, so does $\Gamma$.

Finally let us also show the further statement in iii) that when $\Sigma_t$ contains $\mathcal{A}$, the multiplicity $n$ of $\Gamma$ is an odd integer.  Fix again a ball $B$ centered about a point $p\in\mathcal{A}$.  Taking replacement for $\Sigma_i$ in $B$, one obtains a $\mathbb{Z}_2$-stable surface $V_j$.  We have seen that $V_j$ converges to $\mathcal{V}$ smoothly away from finitely many points.  Choose a point $p\in B$ centered around the axis $\mathcal{A}$ and a ball $B_1$ centered around it that avoid these finitely many points.  Then in $B_1$, $V_j$ consists of $m$ graphs $f^j_1,..., f^j_m$ each converging smoothly to $\mathcal{V}$ (as $j\rightarrow\infty$).  Precisely one of the graphs $f^j_s$ among the $f^j_1,...f^j_m$ contains the axis $\mathcal{A}$ in $B_1$.  The graph $f^j_s$ is preserved by the involution $\tau$ generating $\mathbb{Z}_2$.  The other graphs, being disjoint from the singular axis, must be swapped one with another.  It follows that $m$ is odd.  Thus iii. is established.

The argument for ii) follows with straightforward modifications of the proof of the Improved Lifting Lemma from the arguments of \cite{ketover}, \cite{ketover2} and we omit it.  

\qed

\end{document}